\documentclass{amsart}

\author{Malte Ha{\ss}ler}
\address{Cornell University, Department of Mathematics, Malott Hall, 212 Garden Avenue, 14853 Ithaca, USA}
\email{mh2479@cornell.edu}

\usepackage[left=3.5cm, right=3.5cm]{geometry}
\usepackage[dvipsnames]{xcolor}
\title{Lower bounds on non-random fluctuations in planar first passage percolation}

\usepackage{amsthm, graphicx}
\usepackage{amsmath, amsthm, amsfonts,amssymb, comment, bbm, enumitem, hyperref}
\newtheorem{Theorem}{}
\newenvironment{maintheorem}[1]{\renewcommand \theTheorem{#1}\Theorem}{\endTheorem}

\newtheorem{lemma}{Lemma}

\newcounter{reminder}

\newcommand{\note}[1]{}
\newcommand{\hide}[1]{}

\theoremstyle{definition}

\newcommand{\eps}{\varepsilon}

\newcommand{\R}{\mathbb R}

\newcommand{\Z}{\mathbb Z}
\newcommand{\E}[2]{\mathbb{E}_{#1} #2 }
\newcommand{\Prob}[1]{\mathchoice
{\mathbb P \left( \rule{0pt}{11pt} #1 \right)}
{\mathbb P \left(#1\right)}
{\mathbb P \left(#1\right)}
{\mathbb P \left(#1\right)}
}

\begin{document}
\maketitle

\begin{abstract}
The fluctuations of the passage time in first passage percolation are of great interest. We show that the non-random fluctuations in planar FPP are at least of order $\sqrt{\log n}$ under some conditions that are known to be met for a large class of absolutely continuous edge weight distributions. This improves the ${\log(\log(n))}$ bound proven by Nakajima and is the first result showing divergence of the fluctuations for arbitrary directions. 

Our proof is an application of recent work by Dembin, Elboim and Peled on the BKS midpoint problem and the development of Mermin-Wagner type estimates.
\end{abstract}
\medskip

\section{Introduction}

First passage percolation, as initially introduced by Hammersley and Welsh \cite{hamwelsh}, is a simple model about a random metric on a graph: Consider the integer lattice $\Z^d$ and identically and independently put weights on the edges of the lattice according to some probability measure $\nu$ supported on $[0,\infty)$. The \emph{passage time} $T(p)$ of a path $p$ is defined as the sum of the weights of all edges of $p$. The passage time $T(x,y)$ between two vertices is the infimum of $T(p)$ taken over all paths connecting $x$ and $y$. Under some mild assumptions on $\nu$ (which imply that $\E{}{T(x,y)}<\infty$), subadditivity of passage times yields that $T(0,xn)/n$ converges almost surely to a deterministic time constant $\mu(x)$, which depends on $\nu$. Moreover, the rescaled, random balls $\{x \in \Z^d: \, T(0,x)\le t\}/t$ converge to a deterministic, convex, non-empty limit shape \cite{coxdurrett}.   

Major questions concern the nature of the limit shape, the behavior of geodesics and the fluctuations of the passage times. We will focus on the latter and the planar $d=2$ case. The fluctuations are usually treated in a separate random and non-random part, respectively:
\begin{equation}
\label{fluc}
T(0,x)-\mu(x)=\left( T(0,x)-\E{}{T(0,x)} \right)+ \left(\E{}{T(0,x)}-\mu(x)\right).
\end{equation}

First passage percolation is believed to reside in the so-called KPZ universality class. Thus, an ample amount of detailed conjectures exist about the model. In particular, both summands in \eqref{fluc} should be of order $|x|^{1/3}$. Yet, this is far from achievable for any $\nu$ without additional assumptions. The $\log n$ lower bound on the variance of the passage time by Newman and Piza \cite{newmanpiza} from the 90s is still the best available. Damron, Hanson, Houdré and Xu \cite{damronlb} later also showed that the mean absolute deviation $|T(0,x)-\E{}{T(0,x)}|$ is typically at least of order $\sqrt{\log n}$. For the non-random fluctuations, the best lower bound is of constant order \cite{nakajima}. In the same paper, Nakajima also gave a $\log\log(n)^{1/d}$ bound for arbitrary dimensions, however it requires a local limit shape assumption that is not verified for explicit directions. 

\subsection{Main result}

In order to apply the midpoint theorem in \cite{BKS} we need either of the following assumptions. Here $\eps_{40}>0$ and $\sigma_{40}\in (0,\eps_{40}/2)$ are absolute constants. As mentioned in \cite{BKS}, a uniform distribution with small support around $1$ is an example of a distribution satisfying \eqref{support}. 

\begin{multline}
\text{The distribution } \nu \text{ has finite exponential moments } \E{}{e^{tX}} \text{ for some } X\sim\nu, \, t>0 \\ \text{and the corresponding limit shape is not a polygon with } 40 \text{ sides or less}.
\tag{A1}
\label{expsides}
\end{multline}

\begin{equation}
\text{The distribution } \nu \text{ is supported on } [1,1+\eps_{40}] \text{ with variance at least } \sigma_{40}^2.
\tag{A2}
\label{support}
\end{equation}
\medskip
\begin{maintheorem}{Main Theorem}\label{Thm:nrfluclb}
Let $\nu$ be an absolutely continuous probability measure $\nu$ satisfying \eqref{expsides} or \eqref{support}. There exists a constant $c>0$ such that for all $x\in \Z^2\setminus \{0\}$,
\[
\E{}{T(0,x)}-\mu(x)\ge c \sqrt{\log |x|}. 
\]
\end{maintheorem}

\subsection{Idea of the proof} We start with a surprisingly simple idea by Nakajima \cite[eq.~(3.1)]{nakajima} used in showing that the non-random fluctuations are at least of constant order. Since $\E{}{T(-x,x)}\ge \mu(2x)$ for any $x\in \Z^2$ by subadditivity, 
\begin{equation}
\label{eq:goal}
2(\E{}{T(0,x)}-\mu(x))\ge \E{}{\left[ T(-x,0)+T(0,x)-T(-x,x)\right]}.
\end{equation}
Next, we rely on the methods used by Dembin, Elboim and Peled \cite{BKS}. Firstly, the geodesic from $-x$ to $x$ will likely not intersect a small region around the origin. Secondly, increasing the weights in this region will retain a plausible event that does not affect the geodesic from $-x$ to $x$.

\subsection{Outlook}

We mentioned Damron et. al. \cite{damronlb} and their lower bound on the absolute deviation of the passage time. If their result still holds in a local sense when only edges with distance at most $|x|^{1/16}$ from the origin are changed, then one can reproduce the following: If the absolute deviation is at least $f(n)$ in this local sense, then the non-random fluctuations are at least $f(n^{1/16})$. While similar to the $'\gamma \ge \chi '$-result \cite{adhgammachi}, no strong existence of the variance exponent $\chi$ is required here. By applying a recent theorem by Dembin and Elboim \cite{defluc} extending the older result in \cite{newmanpiza}, it would follow that non-random fluctuations are at least of order $|x|^{1/128}$ under a uniform curvature assumption. 

Weakening the moment and continuity assumptions is conceivable as noted in \cite{BKS}. We would also like to relax \eqref{support} but this appears to be difficult at least in our undirected model. In theory, the limit shape of any non-deterministic probability measure should not be a polygon, yet this has not been shown for any absolutely continuous measure. For higher dimensions, our proof only yields a constant bound, so new methods need to be developed.

\subsection{Notation}

(some specifications are omitted if they are clear within the context)

$|x|:\,\ell_1$ norm of $x$

$\Lambda(n)$: A square of side length $2n$ centered at the origin of $\Z^2$, rounded to nearest integer

$\partial\Lambda(n)$: The discrete boundary of $\Lambda(n)$ (any reasonable definition works as long as $\Lambda(n)$ is the disjoint union of $\partial\Lambda(k)$ for $0\le k\le n$). 

$T(p)(\omega)$: Passage time of a path $p$ in the edge weight environment $\omega$

$T(a,b)(\omega),[T_K(a,b)(\omega)]$: Passage time from a vertex $a$ to a vertex $b$ in the environment $\omega$ [when only paths completely contained in $K$ are considered]. 

$\gamma(a,b)(\omega)$: The geodesic from $a$ to $b$ for the environment $\omega$. Since we work with absolutely continuous weight distributions, we may assume uniqueness of geodesics.

We also abbreviate by setting $T(-n,n)=T((-n,0),(n,0))$, $T(0,n)=T((0,0),(n,0))$ and similarly for $T(-n,0),\gamma(-n,n)$ etc. 

\section{Some auxiliary results}

The following lemmas play a key role in the proof of the theorem. The first one is a Mermin-Wagner type argument as found in \cite{BKS}. See also subsequent applications of this lemma in \cite{defluc} and \cite{smallball}. Similar to the original context in physics, this argument shows existence of fluctuations at low cost, where low cost here means occurring at uniformly positive probability. 

\begin{lemma}[\cite{BKS}, Lemma 2.12]
\label{Lem:mermin-wagner}
Let $\nu$ be an absolutely-continuous probability measure on $\R$. There exist
\begin{itemize}
\item a Borel set $S \subset \R$ with $\nu(S)=1$,
\item Borel subsets $(B_\delta)_{\delta>0}$ of $S$ with $\lim_{\delta \to 0} \nu(B_\delta)=1$,
\item for each $\sigma\in [0,1]$, increasing bijections $g_\sigma: S \to S$
\end{itemize}
such that the following holds:
\begin{enumerate}
\item For $\sigma\in [0,1]$ and $\delta>0$,
\[
g_\sigma(w)\ge w+\delta \sigma \text{ for } w \in B_\delta
\]
and $g_0(w)=w$.
\item Given an integer $n\ge 1$ and a vector $\tau=(\tau_1,\dots,\tau_n)\in [0,1]^n$, define the bijection $T_\tau: S^n \to S^n$ component-wise by
\[
T_\nu(w)_i=g_{\tau_i}(w_i).
\]
Then for each Borel set $A \in \R^n$,
\[
\nu^n(T_\tau(A))\ge e^{-\| \tau \|^2} \nu^n(A)^2
\]
where $T(A):=\{T(a):  a \in A \cap S^n\}$. 
\end{enumerate}
\end{lemma}

The second lemma relates to the BKS midpoint problem, addressing the probability that a geodesic goes through the origin. 

\begin{lemma}[midpoint problem]
\label{Lem:midpoint}
For any absolutely continuous probability measure with finite exponential moments and a limit shape that is not a polygon with $40$ sides or less,
\[
\lim_{n \to \infty} \Prob{ \Lambda(n^{1/17}) \cap \gamma(-n,n)=\emptyset }=1. 
\]
\end{lemma}
\begin{proof}
See Theorem~1.2 in \cite{BKS}. The probability of the origin (or any nearby vertex) lying on $\gamma(-n,n)$ decreases with order $n^{-1/16}$, so we can take a union bound over all vertices on the boundary of $\Lambda(n^{1/17})$. 
\end{proof}

It is well-known that in subcritical Bernoulli percolation closed edges eventually cannot be avoided. The lemma below expands this idea, saying that jumping over a few (= sublinearly many) closed edges will not suffice. \newpage

\begin{lemma}[supercriticality implies positive ratio]
\label{Lem:positiveratio}
Let $B \subset \R$ be a measurable set with $\nu(B)\ge 1-p_c(d)$. Then there exist constants $a,C,c>0$ such that for all $n$ we have
\[
\Prob{\forall \gamma \in \Gamma(0,\partial\Lambda(n)): \, \sum_{e\in \gamma} 1(\tau(e)\in B)\ge an }\ge 1-C \exp(-cn),
\]
i.e.\ with high probability, at least $an$ many edges of a geodesic from the origin to $\partial\Lambda(n)$ assume a weight in $B$. 
\end{lemma}
\begin{proof}
Consider an FPP model with Bernoulli edge weights $\Prob{\tau(e)=1}=\Prob{B}$ and $\tau(e)=0$ otherwise. Since $\nu(B^c)<p_c(d)$, there exists $a>0$ such that $\mu(x)\ge 2a$ for all $x$ with $|x|_\infty=1$ (this is a consequence of the limit shape theorem). Using a large deviation result \cite[Theorem~5.2]{kesten} and a union bound, the passage time from the origin to any point in $\Lambda(n)$ is at least $a n$ with exponentially high probability. This means that any geodesic, and hence any path from the origin to $\Lambda(n)$, must contain at least $an$ many edges with weight $1$. Translating this result to the original model finishes the proof. 
\end{proof}

\section{proof of main result}

We will stick to the case $x=ne_1$. This simplifies notation, the proof is similar for arbitrary directions. All constants used can be chosen uniformly for all directions, essentially due to compactness of the limit shape. See Figure~\ref{Fig:A} for a visual summary of the argument.

Define the following events, depending on $n$. Here we choose $\delta>0$ small enough such that $\nu(B_\delta)>1/2$ where $B_\delta$ is given by Lemma~\ref{Lem:mermin-wagner}. 

$A_1$: The geodesic from $(-n,0)$ to $(n,0)$ does not intersect $\Lambda(n^{1/17})$.

$A_2$: Any geodesic between points in $\Lambda(n)$ is contained in $K=\Lambda(Cn)$ for some $C>1$ specific to $\nu$

$A_3(\delta)$: Any path between points $x,y\in\Lambda(n)$ with $|x-y|\ge n^{1/34}$ contains at least $a \cdot |x-y|$ many edges in $B_\delta$ (the constant $a$ is provided by Lemma~\ref{Lem:positiveratio})

All of these events have high probability for large $n$: This follows from Lemma~\ref{Lem:midpoint} for $A_1$, a standard large deviation result for $A_2$ and Lemma~\ref{Lem:positiveratio} for $A_3$. A $K$-geodesic is the passage time-minimizing path among all paths contained in $K$. We define the events $A_1^K,A_3^K\subset \R^K$ similarly to $A_1$ and $A_3$ but with the geodesic replaced by the $K$-geodesic and only paths contained in $K$ are considered. Clearly, those events also have a high probability and the events coincide with $A_1$ and $A_3$, respectively, if conditioned on $A_2$. 

Define the event  $A=A_1^K\cap A_3^K \cap S^K$. Then $\nu^K(A)\ge 3/4$ (or any value in $(0,1)$) for sufficiently large $n$. We now turn to the construction of the weight modifiers $\tau_e\in [0,1]$ for $e\in K$. 

\[
\tau_e=
\begin{cases} k^{-1} \log(n)^{-1/2} &\text{ if } e \in \partial\Lambda(k) \text{ for some } k \in \{2,...,n^{1/17}\} \\
0 &\text{ otherwise }.
\end{cases}
\]
Note that for some absolute constant $C_1$
\begin{equation}
\|\tau\|^2=\sum_e \tau_e^2=\sum_{k=2}^{n^{1/17}}  \frac{8k}{k^2 \log(n)}\le C_1. 
\end{equation}
We are now ready to apply Lemma~\ref{Lem:mermin-wagner} with our choice of $\tau$ and $A$. The lemma then gives us a bijection $T: S^K \to S^K$ such that
\[
\nu^K(T(A))\ge e^{-\|\tau\|^2}\nu^k(A)^2 \ge e^{-C_1}(3/4)^2=:c_2>0.
\]
Here $c_2$ does not depend on $n$. Define the event
\[
M_n:\{ \omega: \omega\in A_2\cap A_1, \omega_{|K}\in T(A) \}
\]
Then for large $n$, $\Prob{M_n}\ge c_2/2$. For each edge weight configuration $\omega \in M_n$, define 
\[
\tilde\omega=
\begin{cases}
    \omega \text{ on } K^c \\
    T^{-1}(\omega) \text{ on } K. 
\end{cases}
\]
Since $g_\sigma$ is increasing, we have $\tilde\omega \le \omega$ on $\Lambda(n^{1/17})$ and equal weights outside. We claim the following:

\textbf{Claim 1:} $T_K(0,n)(\omega)\ge T_K(0,n)(\tilde \omega)+c\log(n)^{1/2}$
\begin{proof}
By Lemma~\ref{Lem:mermin-wagner}, for any path $p$ from $0$ to $\partial\Lambda(n^{1/17})$, and $e\in p$ we have $\omega(e)\ge \tilde{\omega}(e) $ and if $\tilde\omega(e)\in B_\delta$, we even have $\omega(e)\ge \tilde{\omega}(e) + \delta \tau_e$. The path $p$ contains disjoint sub-paths $p_1,...,p_{n^{1/34}}$ such that $p_i$ connects $\partial\Lambda((i-1)n^{1/34})$ with $\partial\Lambda(i n^{1/34})$. Since $\tilde\omega \in A \subset A_3^K$, every path $p_i$ contains at least $an^{1/34}$ many edges in $B_\delta$. Hence,
\begin{align*}
&\sum_{e\in p} \tau_e 1(\tilde\omega(e)\in B_\delta)
\ge a n^{1/34} \sum_{i=1}^{n^{1/34}-1} \frac{1}{i n^{1/34}\log(i n^{1/34})^{1/2}} \\
&\ge a n^{1/34} \int_1^{n^{1/34}} \frac{1}{x n^{1/34}\log(x n^{1/34})^{1/2}}\,dx 
 =  a \int_{n^{1/34}}^{n^{1/17}} \frac{1}{y \log(y)^{1/2}} \, dy \\
&=(a/(1/2)) \left[\log(n^{1/17})^{1/2}-\log(n^{1/34})^{1/2}\right]\ge c_3 \log(n)^{1/2}
\end{align*}
for some $c_3>0$. For $p$ we choose the subpath of $\gamma(0,n)(\omega)$ until the first exit from $\Lambda(n^{1/17})$ and let $p^c$ denote the remaining path to $(n,0)$. We then obtain
\begin{align*}
&T(0,n)(\omega)=T(p)(\omega)+T(p^c)(\omega)=T(\gamma(0,n)(\omega))(\tilde \omega)+T(p)(\omega)-T(p)(\tilde\omega)\\
&\ge T(\gamma(0,n)(\omega))(\tilde\omega)+\delta \sum_{e\in p} \tau_e 1(\tilde\omega(e)\in B_\delta) \ge T(0,n)(\tilde \omega)+ c_3 \log(n)^{1/2}.
\end{align*}
\end{proof}
\textbf{Claim 2:} $T_K(-n,n)(\omega)=T_K(-n,n)(\tilde \omega)$
\begin{proof}
Since $\omega$ and $\tilde\omega$ are both in $A_1^K$, both $K$-geodesics $\gamma_K(-n,n)(\omega)$ and $\gamma_K(-n,n)(\tilde\omega)$ lie in a region where the edge weights of $\omega$ and $\tilde\omega$ coincide. The respective passage times must then be equal by definition of a geodesic.  
\end{proof}

We now combine the claims. In the first line we use $\omega \in A_2$ and Claim 2. In the following inequality we use $\omega\ge \tilde\omega$ and Claim 1. Finally, we use subadditivity of the passage time $T_K$.

\begin{align*}
&T(-n,0)(\omega)+T(0,n)(\omega)-T(-n,n)(\omega)
=T_K(-n,0)(\omega)+T_K(0,n)(\omega)-T_K(-n,n)(\tilde \omega) \\
&\ge T_K(-n,0)(\tilde\omega)+T_K(0,n)(\tilde \omega)+c_3 \log(n)^{1/2}-T_K(-n,n)(\tilde \omega) \ge c_3 \log(n)^{1/2}.
\end{align*}

\medskip
Since the event $M_n$ occurs with a positive probability uniformly bounded away from zero, 
\begin{equation*}
\E{}{\left[ T(-n,0)+T(n,0)-T(-n,n)\right]}\ge  \frac{c_2c_3}{2} \log(n)^{1/2}.
\end{equation*}
Applying \eqref{eq:goal} finishes the proof.

\begin{figure}
\centering
\includegraphics[scale=0.165]{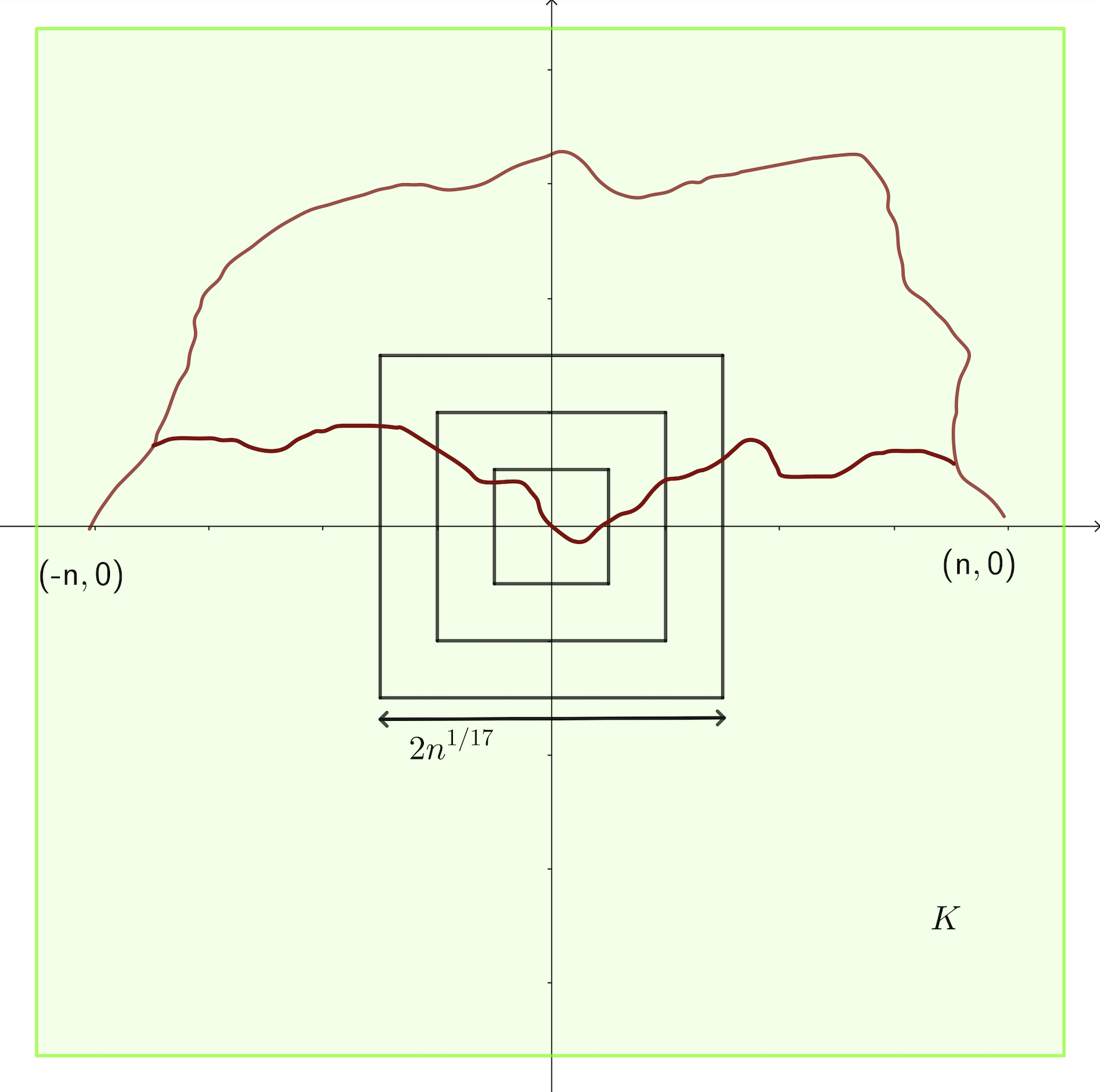}
\caption{
The geodesics from $(-n,0)$ to $(n,0)$ are the same for both environments $\omega$ and $\tilde\omega$ because they do not intersect $\Lambda(n^{1/17})$. But for geodesics going through the origin, the passage time for $\omega$ is larger. Since $\tilde\omega \in A_3^K$, the geodesic $\gamma(n,0)(\omega)$ must go through edges $e$ in each of the rectangular annuli of width $n^{1/34}$ where the edge weight is at least $\tau_e$ larger compared to $\tilde \omega$.}
\label{Fig:A} 
\end{figure}

\medskip

\end{document}